\documentclass[12 pt]{article}
\usepackage{amsmath, amsthm, amssymb,enumerate,verbatim,authblk}

\usepackage{tikz}
\usetikzlibrary{shapes}
\usetikzlibrary{calc}
\usetikzlibrary{positioning}
\usepackage{fullpage}
\usepackage{hyperref}

\usepackage[includeall,
            vmargin={1in,1in},
            hmargin={1in,0.375in}
            ]{geometry}

\title{Bounding the number of hyperedges in friendship $r$-hypergraphs} 
\author{Karen Gunderson\thanks{Heilbronn Institute for Mathematical Research, Department of Mathematics, University of Bristol, U.K. Email: karen.gunderson@bristol.ac.uk} \hspace{.7cm}
Natasha Morrison\thanks{Mathematical Institute, University of Oxford, Woodstock Road, Oxford OX2 6GG, UK. \\Email: morrison@maths.ox.ac.uk} \hspace{.7cm}
Jason Semeraro\thanks{Heilbronn Institute for Mathematical Research, Department of Mathematics, University of Bristol, U.K. Email: js13525@bristol.ac.uk}}
\newtheorem{thm}{Theorem}
\newtheorem{lem}[thm]{Lemma}
\newtheorem{prop}[thm]{Proposition}
\newtheorem{cor}[thm]{Corollary}
\newtheorem{ques}[thm]{Question}

\newtheorem*{thm-main1}{Theorem \ref{thm:main}}

\theoremstyle{definition}

\newcommand{\calH}{\mathcal{H}}
\newcommand{\R}{\mathcal{R}}

\begin{document}

\maketitle

\begin{abstract}
For $r \ge 2$, an $r$-uniform hypergraph is called a \emph{friendship $r$-hypergraph} if every set $R$ of $r$ vertices has a unique `friend' -- that is, there exists a unique vertex $x \notin R$ with the property that for each subset $A \subseteq R$ of size $r-1$, the set $A \cup \{x\}$ is a hyperedge. 

 We show that for $r \geq 3$, the number of hyperedges in a friendship $r$-hypergraph is at least $\frac{r+1}{r} \binom{n-1}{r-1}$, and we characterise those hypergraphs which achieve this bound. This generalises a result given by Li and van Rees in the case when $r = 3$.
 
 We also obtain a new upper bound on the number of hyperedges in a friendship $r$-hypergraph, which improves on a known bound given by Li, van Rees, Seo and Singhi when $r=3$.
\end{abstract}

\section{Introduction}\label{sec:intro}

The Friendship Theorem, first proved by Erd\H{o}s, R\'{e}nyi, and S\'{o}s \cite{ERV66}, states that if $G$ is a graph with the property that every pair of vertices have a unique common neighbour, then $G$ consists of triangles all sharing a single common vertex.  That is, $G$ has a vertex $v$ adjacent to all others and the graph induced by $V(G)\backslash \{v\}$ is a matching.  Such a graph is sometimes called a `windmill graph'.  Subsequently, many different proofs of the Friendship Theorem have been given \cite{cH02, LP72, hW71}.

The Friendship Theorem was first generalised to hypergraphs by S\'{o}s \cite{vS76} who used certain combinatorial designs to construct $3$-uniform hypergraphs with the property that for every $3$ vertices, there is a unique `friend'; a vertex contained in hyperedges with all pairs among the $3$ vertices.  Such hypergraphs are called \emph{friendship $3$-hypergraphs} or \emph{$3$-uniform friendship hypergraphs}.

For any $r \geq 3$, an $r$-uniform hypergraph $\mathcal{H}$ and a set of vertices $A$, say that a vertex $u \notin A$ is a \emph{friend of $A$ in $\mathcal{H}$} if and only if for every set $B \subseteq A$ with $|B| = r-1$, there is a hyperedge $B \cup \{u\} \in \mathcal{H}$.  An $r$-uniform hypergraph $\mathcal{H}$ is called a \emph{friendship $r$-hypergraph} (or \emph{$r$-uniform friendship hypergraph}) if and only if every set of $r$ vertices has a unique friend in $\mathcal{H}$.  J{\o}rgensen and Sillasen \cite{JS14} first introduced the study of friendship $r$-hypergraphs for general $r$, which is the focus of this paper. 

One can naturally obtain an analogue of the windmill graphs for all $r \geq 2$; these graphs can be constructed from Steiner systems, as was noted by S\'{o}s \cite{vS76} in the case $r=3$ and for arbitrary $r$ by J{\o}rgensen and Sillasen \cite{JS14}. An $S(t, k, n)$ \emph{Steiner system} is a $k$-uniform hypergraph on $n$ vertices with the property that every collection of $t$ vertices appears together in exactly one hyperedge.  A celebrated result of Keevash \cite{pK14} states that $S(r-1, r, n)$ Steiner systems exist for every $n$ sufficiently large satisfying a set of divisibility conditions (depending on $r$). Note that an $S(r-1, r, n)$ Steiner system is precisely an $r$-uniform hypergraph in which every set of $r-1$ vertices has a unique friend.  These can be used to produce examples of friendship $r$-hypergraphs, as described in the following.

Let $\mathcal{S}$ be an $S(r-1,r,n-1)$ Steiner system and form a new $r$-uniform hypergraph $U_r(\mathcal{S})$ with vertex set $V(\mathcal{S}) \cup \{u\}$, and edge set
\begin{equation*}\label{eq:universal}
E(\mathcal{U}_r(\mathcal{S})) = E(\mathcal{S}) \cup \{ A \cup \{u\} \mid A \subseteq V(\mathcal{S}) \text{ with } |A| = r-1\}.
\end{equation*}

It is easy to see that  $U_r(\mathcal{S})$ is a friendship $r$-hypergraph. A friendship $r$-hypergraph $\calH$ is called \textit{universal} if there exists some Steiner system $\mathcal{S}$ with the property that $\calH \cong  U_r(\mathcal{S})$. Note that as there may be more than one $S(r-1,r,n-1)$ Steiner system, for a fixed $n$, such hypergraphs need not be unique.  If $\mathcal{H} \cong U_r(\mathcal{S})$ with vertex $u \notin V(\mathcal{S})$, then $u$ is called a \emph{universal vertex in $\mathcal{H}$}.  A simple example of a family of universal friendship $r$-hypergraphs on any $r \ge 3$ is given by the complete hypergraphs $K_{r+1}^r$, whose edges consist of all subsets of size $r$ of a set of $r+1$ vertices. 

S\'{o}s \cite[\S 2.5]{vS76} asked, in the case $r = 3$, if one could identify all friendship $3$-hypergraphs.  Hartke and Vandenbussche \cite{HV08} found, by computer search, examples of non-universal friendship $3$-hypergraphs on $8$, $16$, and $32$ vertices and proved that for $n \leq 10$, the only non-universal friendship $3$-hypergraph is the one they gave on $8$ vertices.  Li, van Rees, Seo, and Singhi \cite{LRSS12} gave new upper and lower bounds on the size of friendship $3$-hypergraphs and, using a computer search, showed that there are no friendship $3$-hypergraphs on $11$ or $12$ vertices.  They further studied `geometric' friendship $3$-hypergraphs - those which can be embedded in a Boolean quadruple system - and showed that the three friendship $3$-hypergraphs on $16$ vertices found by Hartke and Vandenbussche \cite{HV08} are the only such friendship $3$-hypergraphs on $16$ vertices.

J{\o}rgensen and Sillasen \cite{JS14} gave the first example of an infinite class of non-universal friendship $3$-hypergraphs.  Their example, which they call \emph{the cubeconstructed hypergraph} is as follows.  For any $k \geq 3$, define a $3$-uniform hypergraph on vertex set $\{0,1\}^k$ with hyperedges consisting of all sets $\{\mathbf{x}, \mathbf{y}, \mathbf{z}\}$ with the property that for each $i \in [k]$, the coordinates $x_i, y_i, z_i$ are neither all $0$ nor all $1$.  Alternatively, these are the triples with the property that $||\mathbf{x} - \mathbf{y}||_1 + ||\mathbf{y} - \mathbf{z}||_1 + ||\mathbf{z} - \mathbf{x}||_1 = 2k$.  Such hypergraphs have $2^k$ vertices and $2^{k-1}(3^{k-1} - 1)$ hyperedges. In the same paper, J{\o}rgensen and Sillasen \cite{JS14} gave an example, constructed using a Steiner system, of a non-universal $4$-uniform friendship hypergraph on $9$ vertices with $90$ hyperedges.

In this paper we are concerned with finding new bounds for the number of hyperedges in a friendship $r$-hypergraph. Several results in this direction are already known in the case $r = 3$. Li and van Rees \cite{LvR13} proved that a friendship $3$-hypergraph on $n$ vertices has at least $\frac{4}{3}\binom{n-1}{2}$ hyperedges and showed that a hypergraph with exactly this many hyperedges is universal.  Our first result generalises this lower bound and characterisation of the extremal examples to all $r\geq 3$.

\begin{thm}\label{thm:main}
Let $r \ge 3$ and let $\mathcal{H}$ be a friendship $r$-hypergraph on $n \ge r+1$ vertices. Then,
$$|E(\mathcal{H})| \ge \frac{r+1}{r}\binom{n-1}{r-1}.$$
Moreover, equality holds if and only if $\calH$ is universal. 
\end{thm}

One particular way in that our proof differs from that of Li and van Rees \cite{LvR13} is that we use hypergraph saturation results of Bollob\'{a}s \cite{bB65} in two different ways - first to simplify a step in the proof of the lower bound, and, more centrally, to prove that this bound is uniquely realised by the universal graphs.  The proof of the lower bound is closely related to the proof given in \cite{LvR13} in the case $r = 3$ while our proof that the extremal examples are precisely the universal friendship hypergraphs uses a different technique.

For upper bounds on the size of a friendship $3$-hypergraph, Li, van Rees, Seo and Singhi \cite{LRSS12} showed that if $\mathcal{H}$ is a friendship $3$-hypergraph, then
\begin{equation}\label{eq:LRSS-ub}
|E(\calH)| \leq \binom{n}{3}  \frac{2(n-3)}{3n-10} = \frac{n^3}{9}(1+o(1)).
\end{equation}

We give an improved upper bound on the number of hyperedges in a friendship $r$-hypergraph that holds for all $r \geq 3$.

\begin{thm}\label{thm:main2}
Let $r \geq 3$ and let $\mathcal{H}$ be a friendship $r$-hypergraph on $n \geq r+1$ vertices.  Then,
\[
|E(\calH)| \leq \frac{2}{r(r+1)}\left\lfloor \frac{(r+1)(3r-4)}{6}\right\rfloor \binom{n}{r} + \frac{4}{r^2(r+1)}\left\lceil \frac{2(r+1)}{3} \right\rceil \binom{n}{r-1}.
\]
Moreover, if equality holds, then $\calH$ has the property that every set of $r-1$ vertices is contained in the same number of hyperedges.
\end{thm}

In the special case $r = 3$, Theorem \ref{thm:main2} shows that a friendship $3$-hypergraph, $\mathcal{H}$, on $n$ vertices has
\[
|E(\mathcal{H})| \leq \frac{n^2(n-1)}{12} = \frac{n^3}{12}(1+o(1)),
\]
which improves the upper bound given in \eqref{eq:LRSS-ub} by a constant factor.

Furthermore, Theorem \ref{thm:main2} implies that if $\mathcal{H}$ is  a friendship $3$-hypergraph with $|\mathcal{H}| = \frac{n^2(n-1)}{12}$, then every pair of vertices of $\calH$ occurs in the same number of hyperedges.  Li, van Rees, Seo and Singhi \cite{LRSS12} conjectured that there are no friendship $3$-hypergraphs satisfying this condition, except for the trivial case $\mathcal{H} = K_4^3$.  

Using a more careful version of Theorem \ref{thm:main2}, in Section \ref{sec:ex}, we show that the example of a friendship $4$-hypergraph on $9$ vertices with $90$ hyperedges given in \cite{JS14} is as large as possible among friendship $4$-hypergraphs on $9$ vertices.

The remaining sections of the paper are organised as follows.  In Section \ref{sec:lb}, we give the proof of Theorem \ref{thm:main}, by first giving a proof of the lower bound and noting a degree property satisfied by hypergraphs attaining the lower bound.  The second part of the theorem is then proved using this degree property together with bounds for hypergraph saturation. In Section \ref{sec:ub}, we find some structural results and use them to prove Theorem \ref{thm:main2}. We conclude the paper with some related open questions in Section \ref{sec:open-ques}.

\section{Lower bounds}\label{sec:lb}
The section begins with some preliminary results required for the proof of the lower bound; we state two observations of  J{\o}rgensen and Sillasen \cite{JS14} and introduce the necessary hypergraph saturation results of Bollob\'{a}s \cite{bB65}. We then prove our first result in two parts; initially we show the lower bound on the number of hyperedges and then use some key ideas from that argument to conclude that the extremal graphs are universal.

The following lemma will be used in several places throughout our proof.
\begin{lem}[J{\o}rgensen and Sillasen \cite{JS14}]\label{lem:decomp}
Let $\mathcal{H}$ be a friendship $r$-hypergraph. Then
\begin{itemize}
\item[(i)] Every set of at most $r-1$ vertices is contained in at least one hyperedge.
\item[(ii)] Every hyperedge is contained in a unique copy of $K_{r+1}^r$. 
\end{itemize}
\end{lem}

Lemma \ref{lem:decomp}(ii) shows that we may partition the edges of $\calH$ into disjoint sets of size $r+1$, each of which forms a $K_{r+1}^r$. We define the \textit{ $K_{r+1}^r$-decomposition} of $\calH$ to be the $(r+1)$-uniform hypergraph $\calH'$ whose edges consist of all subsets of $r+1$ vertices which induce a $K_{r+1}^r$ in $\calH$.  For much of what follows, the $K_{r+1}^r$-decomposition of any friendship hypergraph shall be primarily used.  Observe that
\begin{equation}\label{ob:edges}
|E(\calH)|=(r+1)|E(\calH')|.
\end{equation}

A key feature of our proof is the application of a hypergraph saturation result of Bollob\'{a}s \cite{bB65}, given in Theorem \ref{thm:sat} below. For any $n$, $k$, $\ell$ with $k+\ell \leq n$, let $M(n, k, \ell)$ be the $k$-uniform hypergraph on vertices $\{1, \ldots, n\}$ with edge set $E$ containing all $k$-sets that have non-empty intersection with the set of $\ell$ vertices $\{1, 2, \ldots, \ell\}$.  The hypergraph $M(n, k, \ell)$ contains no copy of $K_{k + \ell}^k$ but the addition of any $k$-set $e \notin E$  creates a copy of $K_{k+\ell}^k$ containing $e$.  Such a graph is called \emph{$K_{k+\ell}^k$-saturated}.

\begin{thm}[Bollob\'{a}s \cite{bB65}]\label{thm:sat}
Let $\mathcal{H}$ be a $k$-uniform hypergraph on $n$ vertices that is $K_{k+\ell}^k$-saturated.  Then,
\[
|E(\mathcal{H})| \geq |M(n, k, \ell)| = \binom{n}{k} - \binom{n-\ell}{k}. 
\]
Furthermore, if equality holds, then $\mathcal{H} \cong M(n, k, \ell)$.
\end{thm}

We are now prepared to prove our first result.

\begin{prop}\label{prop:lb}
Let $\mathcal{H}$ be a friendship $r$-hypergraph on $n$ vertices and let $\mathcal{H}'$ be its $K_{r+1}^r$-decomposition. Then 
\[
|E(\mathcal{H}')| \geq \frac{1}{r}\binom{n-1}{r-1}.
\] 
Furthermore, if equality holds, then every set of $r-1$ vertices is either contained in exactly one hyperedge of $\mathcal{H}'$ or exactly $(n-r+1)/2$ hyperedges of $\mathcal{H}'$.
\end{prop}

\begin{proof}
Fix $a \in V(\mathcal{H})$ and let $\mathcal{H}(a)$ be the neighbourhood hypergraph of $a$: the $(r-1)$-uniform hypergraph with vertex set $V(\mathcal{H})\backslash \{a\}$ whose hyperedges consist of all hyperedges in $\mathcal{H}$ containing $a$, with $a$ omitted. That is, $E(\calH(a)) = \{E\backslash \{a\}: E \in E(\mathcal{H}), a \in E\}$. Let 
\[
\mathcal{R} = {V(\mathcal{H}(a)) \choose {r-1}}= \{A \subseteq V(\mathcal{H}(a)): |A|=r-1\}.
\]
For each $R \in \R \backslash E(\mathcal{H}(a))$ let $u$ be the unique friend of $\{a\} \cup R$ and let $q(R)$ be the unique hyperedge of $\calH'$ containing $\{u\} \cup R$. Define a colouring $f_a$: $\R \rightarrow \calH' \cup \{a\}$ by setting
\[
f_a(R) = 
	\begin{cases}
		a	&\text{if } R \in E(\mathcal{H}(a));\\
		q(R)	&\text{otherwise},
	\end{cases}
\]
for each $R \in \R$. If $R = \{x_1, \ldots, x_{r-1}\} \in \R \backslash E(\mathcal{H}(a))$, satisfies $f_a(R)= q(R) = R \cup \{u,v\}$ where $u$ is the unique friend of $R \cup \{a\}$, then the choice of $u$ implies that for each $i \in \{1,\ldots,r-1\}$, 
\[
R \cup \{a,u\} \setminus \{x_i\} \in \mathcal{H}
\]
and hence $R \cup \{u\} \setminus \{x_i\} \in E(\mathcal{H}(a))$.  That is, on the subgraph of $\calH(a)$ induced by the vertices of $q$, the set $R$ is the last `missing' hyperedge in an $a$-coloured copy of $K_{r}^{r-1}$. In particular, this means that this subgraph is $K_{r}^{r-1}$-saturated. Hence, by Theorem \ref{thm:sat}, the number of $q$-coloured edges in this subgraph is at most $\binom{r}{r-1} = r$ with equality if and only if the $a$-coloured edges form a copy of $M(r+1, r-1, 1)$. By double-counting the sets of size $r-1$ in $V(\mathcal{H}(a))$, we obtain
\begin{align}
\binom{n-1}{r-1}
	& = \sum_{\underset{a \in q}{q \in \mathcal{H}'}} r + \sum_{\underset{a \notin q}{q \in \mathcal{H}'}} \left \vert \left\{ T \in \mathcal{R} \setminus E(\mathcal{H}(a)) : T \text{ is } q\text{-coloured} \right\}\right \vert \notag\\
	& \leq r | \{q \in E(\mathcal{H}') : a \in q\}| + r | \{q \in E(\mathcal{H}') : a \notin q\}| \label{eq:r-1-sets}\\
	& = r |E(\mathcal{H}')|. \notag
\end{align}
Hence, $|\mathcal{H}'| \geq \frac{1}{r} \binom{n-1}{r-1}$, as required for the first statement of the proposition.

To prove the second part, note that if equality holds, then, in particular, equality holds in \eqref{eq:r-1-sets}. We have that for each $q \in E(\mathcal{H}')$ not containing $a$, the hyperedges in the subgraph of $\calH(a)$ induced by the vertices of $q$ are $a$-coloured and every non-hyperedge is $q$-coloured. Thus any non-hyperedge of $\calH(a)$ is contained in exactly one hyperedge of $\mathcal{H'}$, namely $q$. 

Every set of size $r-1$ appears in some hyperedge in $\mathcal{H}'$. If such a set $T$ is not in some hypergraph $\mathcal{H}(a)$ for some $a \not\in T$, then $T$ occurs in exactly one hyperedge in $\mathcal{H}'$.  Otherwise, for every $a \notin T$, $T \in \mathcal{H}(a)$.    This completes the proof of the second part of the proposition.
\end{proof}

If, as in Proposition \ref{prop:lb}, a set $T \subseteq V(\calH)$ of $r-1$ vertices is contained in exactly one hyperedge of $\calH'$ call it \emph{unsociable}. Otherwise, say that $T$ is \emph{sociable}. The following lemma is the final piece used to prove Theorem \ref{thm:main}. 

\begin{prop}\label{prop:soc-sets}
Let $\mathcal{H}$ be a friendship $r$-hypergraph on $n \geq r+1$ vertices and let $\mathcal{F}$ be the $(r-1)$-uniform hypergraph of sociable sets in $\mathcal{H}$. If $|\mathcal{H}| = \frac{(r+1)}{r} \binom{n-1}{r-1}$, then $|E(\mathcal{F})| = \binom{n-1}{r-2}$ and $\mathcal{F} = M(n, r-1, 1)$.
\end{prop}

\begin{proof}
As before, let $\mathcal{H}'$ be the $K_{r+1}^r$-decomposition of $\mathcal{H}$.  The result follows from Theorem \ref{thm:sat} by showing that $\mathcal{F}$ is $K_r^{r-1}$-saturated and that $|E(\mathcal{F})| =\binom{n-1}{r-2}$. If $n=r+1$ then $\calH=K_{r+1}^r$ and the result follows. Thus we may assume that $n > r+1$.

First we show that for each unsociable set $T$, there is a copy of $K_r^{r-1}$ containing $T$ in $E(\mathcal{F}) \cup T$. Let $X = \{x_1, x_2, \ldots, x_{r-1}\}$ be an unsociable set; thus there exists a unique hyperedge $X \cup \{u, v\}$ of $\mathcal{H}'$ containing $X$.  For any $w \notin X \cup \{u, v\}$, consider the unique friend of $X \cup \{w\}$.  Since there is only one hyperedge in $\mathcal{H}'$ containing $X$, the unique friend of $X \cup \{w\}$ is either $u$ or $v$.  Suppose, without loss of generality, that the unique friend is $u$.  Then, $\mathcal{H}$ contains hyperedges
\begin{equation*}\label{eq:unsoc}
\{x_1, x_2, \ldots, x_{r-1}, w, u\} \setminus \{x_i\},\  i \in \{1, \ldots, r-1\}.
\end{equation*}
Moreover, each of these sets lie in distinct hyperedges of $\calH'$, and none occur in the same hyperedge as $X \cup \{w\}$. Therefore there are at least two hyperedges of $\calH'$ containing $\{x_1, x_2, \ldots, x_{r-1}, u\} \setminus \{x_i\}$ and so these are all sociable sets. Indeed, the subgraph of $\mathcal{F}$ induced by the
vertices $\{x_1, x_2, \ldots, x_{r-1}, u\}$ is precisely $K_r^{r-1}\backslash \{x_1, x_2, \ldots, x_{r-1}\}$.

To see that $\mathcal{F}$ is $K_r^{r-1}$-free, suppose for a contradiction that the set of vertices $Y = \{y_1, y_2, \ldots, y_r\}$ induce a copy $K_r^{r-1}$ in $\mathcal{F}$.  That is, for every $i \in \{1,\ldots, r\}$, the set $Y \setminus \{y_i\}$ is sociable.  Then, for every $x \notin Y$ and $i \in \{1,\ldots,r\}$, there is a hyperedge $Y \cup \{x\} \setminus \{y_i\} \in \mathcal{H}$.  This means that $x$ is a friend of $Y$.  By the uniqueness of friends, $n \le r+1$, a contradiction.

Finally we show that $|E(\mathcal{F})| =\binom{n-1}{r-2}$. By double counting all pairs $(T,E)$, where $E \in E(\mathcal{H}')$ and $T$ is an $(r-1)$-subset of $E$, we get:
\[
\frac{(n-r+1)}{2}|E(\mathcal{F})| + \left(\binom{n}{r-1} - |E(\mathcal{F})| \right) = \vert E(\mathcal{H}' )\vert \binom{r+1}{r-1} = \frac{(r+1)}{2}\binom{n-1}{r-1}.
\]
Rearranging this equation yields the desired result.
\end{proof}

We can now complete the proof of Theorem \ref{thm:main} which is re-stated here.  Note that the existence of a universal friendship hypergraph on $n$ vertices implies the existence of an $S(r-1, r, n-1)$ Steiner system , with all of the corresponding divisibility conditions on $n$ in terms of $r$.

\begin{thm-main1}
Let $r \ge 3$, let $\mathcal{H}$ be a friendship $r$-hypergraph on $n \ge r+1$ vertices and let $\mathcal{H}'$ be its $K_{r+1}^r$-decomposition. Then,
$$|E(\mathcal{H}')| \ge \frac{1}{r}\binom{n-1}{r-1}.$$
Moreover, equality holds if and only if $\calH$ is universal. 
\end{thm-main1}

\begin{proof}[Proof of Theorem \ref{thm:main}]
The lower bound immediately follows from Proposition \ref{prop:lb} and the observation in equation \eqref{ob:edges}. Now suppose that $|E(\mathcal{H}')| = \frac{1}{r}\binom{n-1}{r-1}$.  The remaining part of the proof uses Proposition \ref{prop:soc-sets} to show that $\mathcal{H}$ is universal.  Set $V = V(\mathcal{H})$.
Indeed, by Proposition \ref{prop:soc-sets} there is a vertex $u \in V$ such that each set $T$ of $r-1$ vertices of $V$ is sociable if and only if $u \in T$.  

To see that the vertex $u$ is universal in $\mathcal{H}$, note that for every $B \subseteq V \setminus \{u\}$ with $|B| = r-2$, the set $B \cup \{u\}$ is sociable and furthermore, these are the only sociable sets.  Thus, for every $A \subseteq V\setminus \{u\}$ with $|A| = r-1$, $A \cup \{u\} \in \mathcal{H}$.

Consider now the hyperedges within $V \setminus \{u\}$.  It remains to show that these hyperedges form a Steiner system.  For any $(r-1)$-set $B \subseteq V\setminus \{u\}$, $B$ is not sociable and hence occurs in exactly $2$ hyperedges in $\mathcal{H}$.  One of the hyperedges in $\mathcal{H}$ containing $B$ is $B \cup \{u\}$ and so within $V \setminus \{u\}$, there is a unique hyperedge containing $B$.  This shows that $\mathcal{H}[V \setminus \{u\}]$ is an $S(r-1, r, n-1)$ Steiner system, which completes the proof.
\end{proof}

\section{Upper bounds}\label{sec:ub}

In Theorem \ref{thm:ub} to come, we give an upper bound on the number of hyperedges in a friendship $r$-hypergraph for any $r \geq 3$.  In the special case when $r=3$, this shows that the $K_4^3$-decomposition of any friendship $3$-hypergraph has at most $n^2(n-1)/48$ hyperedges, improving the bound in \eqref{eq:LRSS-ub} given by Li and van Rees \cite{LvR13} by a constant factor.

Before proceeding to the proof of the upper bound, some technical lemmas are given to be used later.

\begin{lem}\label{lem:path-cmpts}
Let $G$ be a graph on $n$ vertices with minimum degree $\delta(G) \geq 1$ and the property that any pair of vertices with degree $1$ are not adjacent.  Then,
\[
|E(G)| \geq \left\lceil \frac{2n}{3}\right\rceil.
\]
\end{lem}

\begin{proof}
Let $c(G)$ be the number of connected components in $G$.  Since $\delta(G) \geq 1$, there are no components consisting of a single vertex and since any two vertices of degree $1$ are not adjacent, there are no components of $G$ consisting of a single edge.  Thus, every component of $G$ contains at least $3$ vertices and so $c(G) \leq \lfloor n/3 \rfloor$.  Since each connected component of $G$ contains a spanning tree,
\[
|E(G)| \geq n - c(G) \geq n - \left\lfloor \frac{n}{3} \right\rfloor = \left \lceil \frac{2n}{3} \right \rceil,
\]
which completes the proof.
\end{proof}

\begin{cor}\label{cor:dual-gr}
Let $r \geq 3$ and let $G$ be a graph on $r+1$ vertices with maximum degree $\Delta(G) \leq r-1$.  If $G$ has the property that any two vertices of degree $r-1$ are adjacent, then
\[
|E(G)| \leq \binom{r+1}{2} - \left\lceil \frac{2(r+1)}{3} \right\rceil = \left\lfloor \frac{(r+1)(3r-4)}{6} \right\rfloor.
\]
\end{cor}

\begin{proof}
Let $H$ be the complement of $G$.  That is, $V(H) = V(G)$ and $E(H) = \binom{[r+1]}{2} \setminus E(G)$.  Then, $\delta(H) = r - \Delta(G) \geq 1$ and if $x$ and $y$ are any pair of vertices with $\deg_H(x) = \deg_H(y) = 1$, then $\deg_G(x) = \deg_G(y) = r-1$.  By assumption, $x$ and $y$ are adjacent in $G$ and hence are not adjacent in $H$.  Thus, by Lemma \ref{lem:path-cmpts},
\[
|E(G)| = \binom{r+1}{2} - |E(H)| \leq \binom{r+1}{2} - \left\lceil \frac{2(r+1)}{3} \right \rceil,
\]
as required.
\end{proof}

The following proposition is key to the proofs that follow. We use Corollary \ref{cor:dual-gr} to give estimates on the number hyperedges in the $K_{r+1}^r$-decomposition of a friendship $r$-hypergraph that contains a fixed set of $r-1$ vertices.

\begin{prop}\label{prop:shadow}
Let $r \geq 3$ and let $\mathcal{H}'$ be the $K_{r+1}^r$-decomposition of a friendship $r$-hypergraph.  For each $q  \in E(\mathcal{H}')$ and $z \notin q$,
\[
|\{q' \in \mathcal{H}' \mid z \in q' \text{ and } |q' \cap q| = r-1 \}| \leq \left\lfloor \frac{(r+1)(3r-4)}{6} \right\rfloor.
\]
\end{prop}

\begin{proof}
Fix $q \in E(\mathcal{H}')$ and define a graph $G$ with vertex set $q$ as follows.  Let $\{x, y\}$ be an edge of $G$ if and only if there exists $q' \in E(\mathcal{H}')$ with $\{z\} \cup (q \setminus \{x, y\}) \subseteq q'$.  That is, the complements of edges in $G$ are precisely those sets of size $r-1$ in $q$ that are contained in some hyperedge of $\mathcal{H}'$ with the vertex $z$. Since, for any such set $A$ of size $r-1$, there is at most one hyperedge in $\mathcal{H}'$ containing $\{z\} \cup A$, then,
\[|E(G)|=
|\{q' \in \mathcal{H}' \mid z \in q' \text{ and } |q' \cap q| = r-1 \}|.
\]

Our goal is to show that Corollary \ref{cor:dual-gr} can be applied to the graph $G$. First we show that $\Delta(G) \leq r-1$. Suppose for a contradiction that there exists some vertex $x_1 \in q$ such that $\deg(x_1) = r$.  Then, for every set $A\subseteq q$ of size $r-1$ that does not contain $x_1$, there is a hyperedge in $\mathcal{H}'$ containing $\{z\} \cup A$.  That is, $z$ is a friend of $q\backslash \{x_1\}$.  But as $q \in E(\mathcal{H}')$, $x_1$ is also a friend of $q \setminus \{x_1\}$, contradicting the uniqueness of the friend $z$.  Thus, $\Delta(G) \leq r-1$.

Now we show that any pair of vertices in $G$ with degree $r-1$ are adjacent. Suppose that there exist $x_1, x_2 \in q$ such that $\deg(x_1) = \deg(x_2) = r-1$ and $x_1 x_2 \notin E(G)$.  Then, $x_1$ and $x_2$ are each adjacent in $G$ to every vertex $x_3, x_4, \ldots, x_{r+1}$.  Hence, for every set $A \subseteq \{x_3, x_4, \ldots, x_{r+1}\}$ with $|A| = r-2$, there is a hyperedge of $\mathcal{H}'$ containing $\{z, x_1\} \cup A$ and a hyperedge of $\mathcal{H}'$ containing $\{z, x_2\} \cup A$.  Since $q \in \mathcal{H}'$, then both $x_1$ and $x_2$ are friends of the $r$-set $\{z, x_3, x_4, \ldots, x_{r+1}\}$, contradicting the uniqueness of friends.  Thus, $x_1 x_2 \in E(G)$.

Thus, Corollary \ref{cor:dual-gr} applies and it follows that $|E(G)| \leq \lfloor (r+1)(3r-4)/6 \rfloor$.
\end{proof}

Proposition \ref{prop:shadow}, together with double counting arguments and Jensen's inequality for convex functions are now used to give an upper bound on the number of hyperedges in a friendship $r$-hypergraph. Theorem \ref{thm:main2} will follow from Theorem \ref{thm:ub} with one observation to be made.

\begin{thm}\label{thm:ub}
Let $r \geq 3$ and let $\mathcal{H}'$ be the $K_{r+1}^r$-decomposition of a friendship $r$-hypergraph on $n$ vertices.  Then,
\[
|E(\mathcal{H}')| \leq \frac{2}{r(r+1)^2}\left\lfloor \frac{(r+1)(3r-4)}{6}\right\rfloor \binom{n}{r} + \frac{4}{r^2(r+1)^2}\left\lceil \frac{2(r+1)}{3} \right\rceil \binom{n}{r-1}.
\]
\end{thm}

\begin{proof}

For any set $\{x_1, x_2, \ldots, x_{r-1}\} \subseteq V(\calH)$, let $\deg(x_1, x_2, \ldots, x_{r-1})$ be the number of hyperedges in $\mathcal{H}'$ containing $\{x_1, x_2, \ldots, x_{r-1}\}$.  Since $\mathcal{H}'$ arises from a friendship $r$-hypergraph we have,
\[
1 \leq \deg(x_1, x_2, \ldots, x_{r-1}) \leq \frac{n-r+1}{2}.
\]
By double counting we get that
\begin{equation}\label{eq:deg-size}
\sum_{x_1, \ldots, x_{r-1} \in V(\calH)} \deg(x_1, x_2, \ldots, x_{r-1}) = \binom{r+1}{r-1} |\mathcal{H}'| = \binom{r+1}{2} |\mathcal{H}'|.
\end{equation}

On the other hand, again by double counting,
\begin{align}
\sum_{q \in E(\mathcal{H}')} &\sum_{x_1, \ldots, x_{r-1} \in q} \deg(x_1, \ldots, x_{r-1}) \notag\\
	& = \sum_{x_1, \ldots, x_{r-1} \in V(\calH)} \deg(x_1, \ldots, x_{r-1})^2 \notag\\
	& \geq \frac{1}{\binom{n}{r-1}} \left(\sum_{x_1, \ldots, x_{r-1}} \deg(x_1, \ldots, x_{r-1}) \right)^2 
	&&\text{(by Jensen's inequality)} \notag\\
	& = \frac{1}{\binom{n}{r-1}}\left( \binom{r+1}{2} |\mathcal{H}'| \right)^2. &&\text{(by eq. \eqref{eq:deg-size})} \label{eq:conv-bd}
\end{align}

Fix any $q \in E(\mathcal{H}')$.  By double counting, since each hyperedge in $\mathcal{H}'$ has $r+1 = (r-1) + 2$ vertices,
\begin{align}
\sum_{\{x_1, \ldots, x_{r-1} \} \subseteq q} \deg(x_1, \ldots, x_{r-1})
	&=\sum_{\underset{|A| = r-1}{A \subseteq q}} |\{q' \mid A \subseteq q'\}| \notag\\
	& = \sum_{\underset{|A| = r-1}{A  \subseteq q}} \left(1 + |\{q' \neq q \mid A \subseteq q'\}| \right) \notag\\
	& = \binom{r+1}{r-1} + \sum_{\underset{|A| = r-1}{A  \subseteq q}}  |\{q' \neq q \mid A \subseteq q'\}|. \label{eq:sum-deg}
\end{align}

The second term in equation \eqref{eq:sum-deg} can be re-written so that Proposition \ref{prop:shadow} can be applied.  Indeed,
\begin{align}
\sum_{\underset{|A| = r-1}{A  \subseteq q}}  |\{q' \neq q \mid A \subseteq q'\}| 
	&=|\{(A, q') \mid A = q' \cap q,\  q \neq q' \in \mathcal{H}'\}| \notag\\
	&=|\{q' \in E(\mathcal{H}') \setminus \{q\} \mid |q' \cap q| = r-1\}| \notag\\
	&= \frac{1}{2}|\{(z, q') \mid z \in q' \setminus q,\ |q' \cap q| = r-1,\ q' \in E(\mathcal{H}')\}|. \label{eq:counting-z}
\end{align}

Now, for each $z \notin q$, by Proposition \ref{prop:shadow},
\begin{equation}\label{eq:use-prop9}
|\{q' \in \mathcal{H}' \mid z \in q' \text{ and } |q' \cap q| = r-1 \}| \leq \left \lfloor \frac{(r+1)(3r-4)}{6} \right\rfloor.
\end{equation}

Therefore, by equations \eqref{eq:sum-deg}, \eqref{eq:counting-z}, and \eqref{eq:use-prop9},
\begin{align}
\sum_{\{x_1, \ldots, x_{r-1} \} \subseteq q} &\deg(x_1, \ldots, x_{r-1}) \notag\\
	& \leq \frac{1}{2}\left((n-r-1)\left\lfloor \frac{(r+1)(3r-4)}{6}\right\rfloor + (r+1) \cdot r \right) \notag\\
	& = \frac{1}{2} (n-r+1)\left\lfloor \frac{(r+1)(3r-4)}{6}\right\rfloor + \left\lceil \frac{2(r+1)}{3} \right\rceil. \label{eq:deg-ub}
\end{align}

Thus, combining equations \eqref{eq:conv-bd} and \eqref{eq:deg-ub} yields
\begin{align*}
\frac{1}{\binom{n}{r-1}} &\binom{r+1}{2}^2 |E(\mathcal{H}')|^2 \\
	& \le \sum_{q \in \mathcal{H}'} \sum_{x_1, \ldots, x_{r-1} \in q} \deg(x_1, \ldots, x_{r-1})
	&&\text{(by eq. \eqref{eq:conv-bd})} \\
	&\le |E(\mathcal{H}')| \left(\frac{1}{2} (n-r+1)\left\lfloor \frac{(r+1)(3r-4)}{6}\right\rfloor + \left\lceil \frac{2(r+1)}{3} \right\rceil \right) 
	&&\text{(by eq. \eqref{eq:sum-deg})}\notag\\
\end{align*}


Rearranging the above expression gives the required bound.
\end{proof}

\begin{proof}[Proof of Theorem \ref{thm:main2}]
Due to the use of Jensen's inequality, if equality holds in Theorem \ref{thm:ub}, then there exists some fixed integer $k$ such that for all sets $\{x_1, \ldots, x_{r-1}\}\subseteq V(\calH)$, $\deg(x_1,\ldots,x_{r-1}) = k$ . Theorem \ref{thm:main2} then follows from Theorem \ref{thm:ub} and the fact that $|E(\mathcal{H})| = (r+1)|E(\mathcal{H}')|$.
\end{proof}

In the case $r = 3$, Theorem \ref{thm:ub} immediately gives the following upper bound for the size of a friendship $3$-hypergraph.

\begin{cor}\label{cor:3-ub}
Let $\mathcal{H}'$ be the $K_4^3$-decomposition of a friendship $3$-hypergraph on $n$ vertices.  Then,
\[
|E(\mathcal{H}')| \leq \frac{1}{8} \binom{n}{3} + \frac{1}{12}\binom{n}{2} = \frac{n^2(n-1)}{48}.
\]
\end{cor}

Peter Allen has pointed out that the hypergraph removal lemma can be used to provide an upper bound that is asymptotically much better than Theorem \ref{thm:main2}.  The hypergraph removal lemma, proved independently by Gowers \cite{wtG06, wtG07} and Nagle, R\"{o}dl, Schacht and Skokan \cite{NRS06, RS04} says that for any $r$-uniform hypergraph $\mathcal{G}$ on $g$ vertices and any $\varepsilon > 0$, there is a $\delta > 0$ so that any $r$ uniform hypergraph on $n$ vertices that contains at most $\delta n^g$ copies of $\mathcal{G}$ can be made $\mathcal{G}$-free by the removal of at most $\varepsilon n^r$ hyperedges.

As was noted in Lemma \ref{lem:decomp}, if $\mathcal{H}$ is a friendship $r$-hypergraph, then every hyperedge is contained in exactly one copy of $K_{r+1}^r$ and so, in particular, the number of copies of $K_{r+1}^r$ in $\mathcal{H}$ is at most $\frac{1}{r+1}\binom{n}{r} = o(n^{r+1})$.  Thus, the hypergraph removal lemma immediately implies that for every $\varepsilon > 0$, there is an $n_0$ so that for all $n \geq n_0$, if $\mathcal{H}$ is a friendship $r$-hypergraph on $n$ vertices, then $|\mathcal{H}| < \varepsilon n^r$.

\subsection{Possible examples}\label{sec:ex}

In this section, we describe a possible generalisation of the construction given by J{\o}rgensen and Sillasen \cite{JS14} of non-universal friendship $r$-hypergraphs.  The construction depends on the existence of certain Steiner systems, which is unknown except in the smallest cases.

Let $r \geq 3$ and suppose that there is a $S(r+1,r+2,2r+4)$ Steiner system, $\mathcal{S}$.  One can form a non-universal friendship $r$-hypergraph from $\mathcal{S}$ as follows.  Fix three vertices $a,b,c \in V(\mathcal{S})$ and let $\calH'$ be the hypergraph on $V(\mathcal{S}) \backslash \{a,b,c\}$ whose hyperedges are given by $$\{B \backslash \{a\} \mid B \in E(S), a \in B, \{b,c\} \notin B\}.$$  J{\o}rgensen and Sillasen \cite{JS14} examined this construction in the case $r = 4$, where it is known that there is a $S(5, 6, 12)$ Steiner system.

For arbitrary $r$, a proof that $\calH'$ is the $K_{r+1}^r$ decomposition of a friendship $r$-hypergraph $\calH$ follows by a natural generalisation of the proof  given by J{\o}rgensen and Sillasen \cite[Theorem 5]{JS14} in the case when $r=4$. Furthermore, an inclusion-exclusion argument shows that
\begin{align*}
|E(\calH')|
	&={2r+3 \choose r} \cdot \frac{1}{r+1}- 2{2r+2 \choose r-1} \cdot \frac{1}{r}+{2r+1 \choose r-2} \cdot \frac{1}{r-1}\\
	&={2r+1 \choose r} \cdot \frac{1}{r+3}.
\end{align*}
 from which one can verify that $\calH$ is universal if and only if $r=2$. The case $r=4$ is also of particular interest, since in this case Theorem \ref{thm:main2} implies that $|E(\calH')| \leq 18$, so that $|E(\calH)| \leq 18\cdot 5 = 90$, which is exactly the number of hyperedges in the example constructed in \cite{JS14} using a $S(5,6,12)$ system. Thus there is a non-universal example which shows the upper bound given by Theorem \ref{thm:main2} is tight. 

In general, it follows from Theorem \ref{thm:main2} that, for $r$ tending to infinity, a friendship $r$-hypergraph on $2r+1$ vertices can have at most 
\[
{2r+1 \choose r} \frac{1}{r+3} \left( 1 + O\left(\frac{1}{r}\right) \right)
\]
 hyperedges.  Thus, if the appropriate Steiner system exists, such an $(r+1)$-uniform hypergraph is asymptotically as large as possible. 

\section{Open questions}\label{sec:open-ques}

As both the lower bound given in Theorem \ref{thm:main} and the upper bound given in Theorem \ref{thm:main2} impose strong conditions on those friendship hypergraphs that attain either of these bounds, there are natural open questions pertaining to the bounds for friendship hypergraphs that are either not universal, or do not have every set of $r-1$ vertices contained in the same number of hyperedges.

By Theorem \ref{thm:main}, every non-universal friendship $r$-hypergraph has at least $r+1$ more hyperedges than a universal friendship hypergraph on the same number of vertices. We wonder if this could be improved:
\begin{ques}  
 How much larger is any non-universal friendship $r$-hypergraph than a universal hypergraph on the same number of vertices? 
\end{ques}

If $\mathcal{H}$ is a friendship $r$-hypergraph in which not all sets of $r-1$ vertices are contained in the same number of hyperedges, then equation \eqref{eq:deg-size} can be strengthened, leading to an incrementally better upper bound. We are interested to see how much more could be achieved:
\begin{ques}
  In general, if $\mathcal{H}$ is a friendship $r$-hypergraph with not all $\deg(x_1, \ldots, x_{r-1})$ equal, how much smaller is the upper bound on $|\mathcal{H}|$ compared to that given in Theorem \ref{thm:main2} when the number of vertices is relatively small?
\end{ques}

\section*{Acknowledgement}
The first author would like to thank Peter Allen for pointing out the application of the Hypergraph Removal Lemma.

\end{document}